\documentclass[11pt]{article}
\usepackage{amsmath,amsthm,amssymb,mathrsfs}
\usepackage{bm,pifont,xcolor, tikz}
\usetikzlibrary{calc}
\usetikzlibrary{positioning, shapes, calc}
\makeatletter

\newcommand{\Rmnum}[1]{\expandafter\@slowromancap\romannumeral #1@}
\newcommand{\scalebinom}[2]{\left(\!\!\begin{array}{c}#1\\#2\end{array}\!\!\right)}

\makeatother

\usepackage{enumitem}
\usepackage{geometry}
\usepackage[colorlinks=true,
linkcolor=black,citecolor=black,
urlcolor=black]{hyperref}

\makeatletter
\def\@seccntDot{.}
\def\@seccntformat#1{\csname the#1\endcsname\@seccntDot\hskip 0.5em}
\renewcommand\section{\@startsection{section}{1}{\z@}%
{18\p@ \@plus 6\p@ \@minus 3\p@}%
{9\p@ \@plus 6\p@ \@minus 3\p@}%
{\large\bfseries\boldmath}}
\renewcommand\subsection{\@startsection{subsection}{2}{\z@}%
{12\p@ \@plus 6\p@ \@minus 3\p@}%
{3\p@ \@plus 6\p@ \@minus 3\p@}%
{\bfseries\boldmath}}
\renewcommand\subsubsection{\@startsection{subsubsection}{3}{\z@}%
{12\p@ \@plus 6\p@ \@minus 3\p@}%
{\p@}%
{\bfseries\boldmath}}
\makeatother

\usepackage{microtype}

\theoremstyle{plain}
\newtheorem{theorem}{Theorem}[section]
\newtheorem{lemma}{Lemma}[section]

\newtheorem{conjecture}{Conjecture}[section]

\theoremstyle{definition}

\newtheorem{case}{Case}

\newtheorem{claim}{Claim}

\allowdisplaybreaks
\title{ Asymptotic stability of clique-saturating edges near Turán  thresholds}
\author{Xianglei Chen$^{a}$, Changxiang He$^{a}$\footnote{Corresponding author:cxhe@usst.edu.cn}\; 
\footnote{This research is supported by the National Natural Science Foundation of China (Nos. 12271362, 12671416).}\\
{\small $^{a}\,$School of Mathematics, Faculty of Basic Sciences, University of Shanghai for Science and Technology,} \\ 
{\small Shanghai 200093, China} }
\date{}

\begin{document}
\maketitle





\begin{abstract}
Let \(f_{p+1}(n,m)\) denote the minimum number of \(K_{p+1}\)-saturating edges in an \(n\)-vertex \(K_{p+1}\)-free graph with \(m\) edges. In this paper, we investigate the stability of this parameter near the Turán thresholds, with particular emphasis on the role of the minimum degree. We introduce the refined parameter \(f_{p+1,p}(n,m,\delta)\). It is defined for   \(K_{p+1}\)-free graphs on $n$ vertices with $m$ edges,  minimum degree \(\delta\), and   containing at least  a copy of \(K_p\). For every fixed integer \(p\ge3\), we establish a stability result when the number of edges is \(\mathrm{ex}(n,K_p)+o(n^2)\), showing that
\[
 f_{p+1,p}  (n,m,\delta)=\left(\frac{(p-2)^2}{2(p-1)^3}+o(1)\right)n^2, \ \text{for }\    m=\mathrm{ex}(n,K_p)+o(n^2), \ \delta\geq \left(\frac{p-2}{p-1}+o(1)\right)n. 
\]
We further obtain a stability theorem in the regime where the number of edges is quadratically smaller than \(\operatorname{ex}(n,K_{p+1})\). More precisely,  for $0 \leq c< \frac{1}{2p(3p-1)}$,
\[ f_{p+1,p}  (n,m,\delta)=\left(\frac{1}{2p}+o(1)\right)n^2,  \ \text{for }\  
  m=\operatorname{ex}(n,K_{p+1})-(c+o(1))n^2, \ \delta\geq \left(\frac{p-1}{p}-2c+o(1)\right)n.  \] 

\end{abstract}

\par\vspace{2mm}

\noindent{\bfseries Keywords:  } clique-saturating edges, Erdős-Tuza  conjecture, stability, extremal graph 
\par\vspace{1mm}

\section{Introduction}

 Extremal problems concerning forbidden subgraphs have played a central role in graph theory.
Given a graph $H$, a graph $G$ is called $H$-{\it free} if it contains no copy of $H$ as a subgraph.
The associated {\it Turán number} $\operatorname{ex}(n,H)$ is defined as the largest possible number of edges in an $n$-vertex $H$-free graph. 
Determining Turán number is one of the fundamental problems in extremal graph theory, originating from the classical result of Mantel \cite{10} and later developed through Turán's work \cite{11}.
 Turán's classical theorem \cite{11} provides a complete characterization of extremal graphs for $K_{p+1}$-free graphs.
 It states that, for any integer $p\ge 2$, every $n$-vertex $K_{p+1}$-free graph contains at most $\operatorname{ex}(n,K_{p+1})$ edges, and equality holds only for the balanced complete $p$-partite {\it Turán graph} $T_p(n)$.


For a graph \(G\), denote by \(\overline G\) its complement. Let \(G\) be a
\(H\)-free graph. An edge \(e\in E(\overline G)\) is called a
\(H\)-{\it saturating edge} of \(G\) if the graph \(G+e\) contains a copy of
\(H\).  This concept arises naturally in Turán-type problems. In particular, a $K_p$-free graph is {\it maximal} $K_p$-free if and only if each of its non-edges is $K_p$-saturating.
While Tur\'{a}n's theorem determines the largest possible number of edges in a $K_p$-free graph of given order,
the corresponding minimization problem for maximal $K_p$-free graphs was studied by Zykov \cite{ Zykov}, Erd\H{o}s et. al \cite{6}. For more references on this minimization problem, we refer interested readers to the recent surveys \cite{4},\cite{8}. For some recent developments on saturation numbers and related variants, see also \cite{2}.

We adopt the following notation from \cite{2023}. For a $K_p$-free graph $G$, let $f_p(G)$ denote the number of $K_p$-saturating edges of 
$G$. we also define $f_p(n, m)$ to be the {\it minimum number of} $K_p$-{\it saturating edges} of an $n$-vertex $K_p$-free 
graph with $m$ edges. For every integer $p \geq 3$,  the Turán graph $T_{p-1}(n)$
implies $f_{p+1}(n,m)=0$ for all $ 0\leq m\leq{\rm ex}(n,K_p)$. 
On the other hand, once $m>\text{ex}(n,K_p)$, we have $f_{p+1}(n,m)>0$.
In particular,
for $p=3$, Erdős and Tuza \cite{Erdos} 
proved that there exists a constant $c>0$ such that $f_4(n,\lfloor \frac{n^2}{4} \rfloor + 1)\geq cn^2$. Thus, exceeding the Tur{\' a}n threshold by only one edge forces the number of $K_4$-saturating edges to increase from zero to quadratic order. Furthermore, Erdős and Tuza conjectured that every \(K_4\)-free graph \(G\) with $\lfloor \frac{n^2}{4} \rfloor + 1$  edges admits at least $\Big(1 + o(1)\Big)\frac{n^2}{16}$  \(K_4\)-saturating edges.
\begin{conjecture}[Erdős--Tuza,  \cite{Erdos}]\label{erdos}
$f_4\left(n, \left\lfloor\frac{n^2}{4} \right\rfloor + 1\right) = \Big(1+o(1)\Big)\frac{n^2}{16}.$
\end{conjecture}
This conjecture was disproved by Balogh and Liu \cite{2014}.  They constructed an $n$-vertex $K_4$-free graph containing $\lfloor \frac{n^2}{4} \rfloor+1$ edges,    which has only 
$\Big(1+o(1)\Big)\frac{2n^2}{33}$ $K_4$-saturating edges. In addition, they proved   that this construction is the best possible.
 \begin{theorem} [Balogh--Liu,  \cite{2014}]
For all integers $ 1\le t\le \frac{n}{66}$, $$f_4\left(n,\left\lfloor \frac{n^2}{4}\right\rfloor+t\right)
=
\Big(1+o(1)\Big)\frac{2n^2}{33}+\Theta(n).$$ 
 \end{theorem}

For arbitrary integer $p\ge 3$, the following result was proven by He et al.\ in \cite{2023}.
\begin{theorem} [{\rm Theorem 1.2}, \cite{2023}]
For all integers $p\geq 3$ and $1\le t\le \dfrac{(p-2)^{3}}{p(p-1)(4p^{2}-11p+8)}n$, $$f_{p+1}\left(n,\ {\rm ex}(n,K_p)+t\right)=\Big(\frac{2(p-2)^2}{p(4p^2-11p+8)}+o(1)\Big)n^2+\Theta(n). $$
 \end{theorem}
This   shows that a  stability phenomenon arises for the  leading-order asymptotic term of \(f_{p+1}\big(n,\operatorname{ex}(n,K_p)\big)\)   when the deviation of the edge count  from $\mathrm{ex}(n,K_p)$ is linear in $n$.
 In \cite{2023},  the authors posed the problem of  determining    $f_{p+1}(n,m)$   for every integer $m$ satisfying $\operatorname{ex}(n,K_p)\leq m \leq \operatorname{ex}(n,K_{p+1})$. In this paper, we investigate the stability of the minimal number of $K_{p+1}$-saturating edges under  a minimum-degree constraint across both linear-scale and quadratic-scale   regimes near $\operatorname{ex}(n,K_p)$ and $\operatorname{ex}(n,K_{p+1})$, and partially answer the question raised in  \cite{2023}.

For graphs with edge number slightly below $\mathrm{ex}(n,K_p)$, Turán's theorem does not ensure the existence of a $K_p$-subgraph, which invalidates the standard Turán-type argument. To overcome this obstacle, we confine our discussion to graphs that necessarily contain a $K_p$.
We define 
 \begin{align*} 
f_{p+1,p}(n,m,\delta)=\min\{ f_{p+1}(G):\ & G \  \text{is a}\ K_{p+1}  \text{-free graph on }\  n \ \text{vertices   with}\  m  \ \text{edges,}\\ 
&\text{ minimum degree} \ \delta \ \text{and     containing at least a copy of}\  K_p\}. 
 \end{align*}
For $m\geq \operatorname{ex}(n,K_{p})+1$, we  abbreviate $f_{p+1,p}(n,m,\delta)$ to $f_{p+1}(n,m,\delta)$.

The main results of this paper establish stability theorems in two different edge-density regimes. The first concerns graphs whose number of edges differs from $\operatorname{ex}(n,K_p)$ by only $o(n^2)$, while the second deals with the regime where the edge number is of  quadratic order in $n$ and close to the corresponding $\operatorname{ex}(n,K_{p+1})$. The corresponding stability results are stated as follows.

\begin{theorem}\label{main} 
Let $n$ be a sufficiently large integer. For all integers $p\geq 3$,
$$ f_{p+1,p}  (n,m,\delta)=\left(\frac{(p-2)^2}{2(p-1)^3}+o(1)\right)n^2, \ \text{for }\ m=\operatorname{ex}(n,K_p)+o(n^2) \ \text{and } \ \delta\geq \Big(\frac{p-2}{p-1}+o(1)\Big)n.$$
\end{theorem}

 Specializing to $p=3$,  and setting the edge number equal to
$\left\lfloor\frac{n^2}{4} \right\rfloor+1$ in the above theorem,    Turán's theorem ensures that $G$ contains at least one copy of $K_p$,  which shows that the Erdős–Tuza Conjecture \ref{erdos} holds under a suitable minimum-degree condition.

\begin{theorem}\label{smain}
Let $c$ be a constant and let $n$ be a sufficiently large integer. For all integers $p\geq 3$ and $0 \leq c< \frac{1}{2p(3p-1)}$,
$$ f_{p+1,p}  (n,m,\delta)=\left(\frac{1}{2p}+o(1)\right)n^2, \ 
\text{for}\    m=\text{ex}(n,K_{p+1})-\left(c+o(1)\right)n^2, \ \text{and}\  \delta\geq \left(\frac{p-1}{p}-2c+o(1)\right)n.$$
\end{theorem}

The paper is structured as follows. 
Sections $2$ and $3$ are devoted to the proofs of  Theorem \ref{main} and Theorem \ref{smain}, respectively. Section $4$ concludes the paper with some remarks.

\section{Stability near $\mathrm{ex}(n,K_p)+o(n^2)$}

In this section, we prove Theorem \ref{main} by establishing asymptotically matching upper and lower bounds. The proof of Theorem \ref{main} is largely based on the arguments in \cite{2014} and \cite{2023}.

 {\bf Proof of the upper bound}\quad 
Consider the upper-bound constructions from \cite{2014} and \cite{2023}, described in detail as follows.  Let 
  $$|V_0|=o(n),\quad|V_i|=\left(\frac{p-2}{(p-1)^2}+o(1)\right)n, \quad |U_j|=\left(\frac{1}{(p-1)^2}+o(1)\right)n,$$
  for $i,j\in[p-1]$, as illustrated in Figure~\ref{E1graph}.

Let $G$ be a graph constructed above. It is not difficult to check that $$ \delta(G)\geq\left(\frac{p-2}{p-1}+o(1)\right)n, \quad  e(G)=\Big(\frac{p-2}{2(p-1)}+o(1)\Big)n^2={\rm ex}(n,K_p)+ o(n^2).$$  The \(K_{p+1}\)-saturating edges of $G $ occur only within the sets
$V_0, V_1,...,V_{p-1}$.
  Hence
$$
f_{p+1}(G )=\sum_{i=0}^{p-1}\scalebinom{|V_i|}{2}=\left(\frac{(p-2)^2}{2(p-1)^3}+o(1)\right)n^2.
$$

\begin{figure}[!h]
 	\centering
 	\scalebox{0.9}{
\begin{tikzpicture}[every node/.style={font=\small}]

\def\n{5} 
\def\gap{1.5}
\def\offset{(\n-1)*0.5*\gap}

\node[circle, draw, thick, minimum size=10pt,inner sep=0pt,
    text width=10pt,
    align=center,
    font=\small] (V0) at (0,4.3) {$V_0$};

\foreach \i in {1,...,\n} {
  \node[
    circle, draw, thick,
    minimum size=28pt,
    inner sep=0pt,
    text width=28pt,
    align=center,
    font=\small
  ] (V\i) at ({(\i-1)*\gap - \offset},2.0) 
  {$V_{\ifnum\i=\n p-1 \else \ifnum\i=\numexpr\n-1\relax p-2 \else \i \fi\fi}$};
}

\foreach \i in {1,...,\n} {
  \node[
    circle, draw, thick,
    minimum size=20pt,
    inner sep=0pt,
    text width=21pt,
    align=center,
    font=\small
  ] (U\i) at ({(\i-1)*\gap - \offset},-0.7) 
  {$U_{\ifnum\i=\n p-1 \else \ifnum\i=\numexpr\n-1\relax p-2 \else \i \fi\fi}$};
}


\node at ($(V2)!0.52!(V3)$) {$\cdots$};
\node at ($(U2)!0.5!(U3)$) {$\cdots$};

\foreach \i in {1,...,\n} {
  \draw[line width=1.2pt] (V0) -- (V\i);
}

\foreach \i in {1,...,\n} {
  \foreach \j in {\i,...,\n} {
    \ifnum\i<\j
      \pgfmathtruncatemacro{\bend}{17 + 7*abs(\i - \j)}
      \draw[line width=1.2pt, bend left=\bend] (V\i) to (V\j);
    \fi
  }
}

\foreach \i in {1,...,\n} {
  \foreach \j in {\i,...,\n} {
    \ifnum\i<\j
      \pgfmathtruncatemacro{\bend}{14 + 6*abs(\i - \j)}
      \draw[line width=1.2pt, bend right=\bend] (U\i) to (U\j);
    \fi
  }
}

\foreach \i in {1,...,\n} {
  \foreach \j in {1,...,\n} {
    \ifnum\i=\j
    \else
      \draw[line width=1.2pt] (V\i) -- (U\j);
    \fi
  }
}

\node[right=6pt of V0] {$o(1)n$};
\node[left=6pt of V1] {$\left(\dfrac{(p-2)}{(p-1)^2} + o(1)\right)n$};
\node[right=6pt of V\n] {$\left(\dfrac{(p-2)}{(p-1)^2} + o(1)\right)n$};

\node[left=6pt of U1] {$\left(\dfrac{1}{(p-1)^2} + o(1)\right)n$};
\node[right=6pt of U\n] {$\left(\dfrac{1}{(p-1)^2} + o(1)\right)n$};

\end{tikzpicture}}
 		\caption{Construction  for the upper bound  of Theorem \ref{main}}
         \label{E1graph}
 	
 \end{figure}
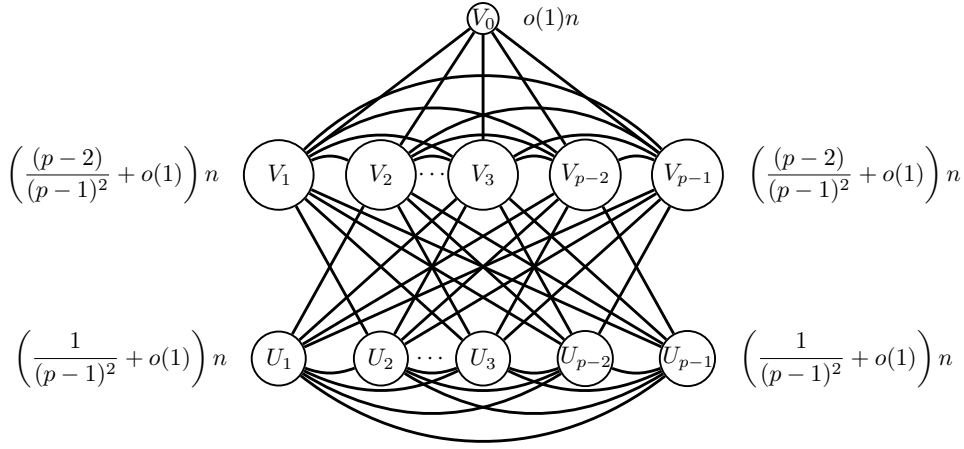

 {\bf Proof of the lower bound}\quad 
 We first note that the proof of the lower bound is essentially adapted from the methods of \cite{2014} and \cite{2023}, with only minor modifications required by the minimum-degree condition.

 Let $G$ be a $K_{p+1}$-free graph with $n$ vertices,  $$e(G)=\left(\frac{p-2}{2(p-1)}+o(1)\right)n^2, \ \delta(G)\geq\left(\frac{p-2}{p-1}+o(1)\right)n,$$ and let it contain a copy of $K_p$.

Following the approach of \cite{2014} and \cite{2023}, denote by $\mathcal{R}=\{ R_{1},R_{2}\ldots,R_{s}\}$ the maximum family of pairwise vertex-disjoint $K_p$'s in graph $G$, where $1\leq s \leq \frac{n}{p}$. Subject to this choice, choose the one that $H_{\mathcal{R}} := G - V(\mathcal{R})$ contains the maximum possible number of edges.
Then $H_{\mathcal{R}}$ is $K_p$-free, and by Tur{\'a}n's theorem,
\begin{align}\label{size}
    e(H_{\mathcal{R}})\leq \text{ex}(n-ps,K_p)=\left(1-\frac{1}{p-1}\right)\frac{(n-ps)^2}{2}+o(n^2).
\end{align}
Moreover, the $K_{p+1}$-saturating edges in $G$ can be classified into two types: those associated with $V(\mathcal{R})$, whose number is $f_1$, and those contained in $H_{\mathcal{R}}$, whose number is $f_2$, so that $f_{p+1}(G)=f_1+f_2$.

For any $v\in V(G)$ 
the neighborhood of the vertex $v$ in graph $G$ is denoted by $N_G(v)$. For any subgraph $ F \subseteq G$, let $N_{F}(v)=N_G(v)\cap V(F)$,  $d_{F}(v)=|N_{F}(v)|$.
For any subset $ L \subseteq V(G)$, define
$$N_G(L) = \bigcap_{v \in L} N_G(v).$$
For a clique $R=\{v_1, v_2,..., v_p\}\in \mathcal{R}$ and any vertex $v\in V(G)\backslash V(R)$, $d_{R}(v)\leq p-1$.
Let $$N_i=N_i(R):=N_{H_{\mathcal{R}}}\left(V(R)\backslash \{v_i\}\right),$$ for every $i\in [p]$. For $1\leq i\leq s$, define 
$$W_i=V(G)\backslash \bigcup_{j\leq i}V(R_j), \ \rho_i=e(R_i, W_i).$$

We now distinguish two cases, depending on whether $s$ is linear in $n$.

\begin{case}
    $s=o(n).$
\end{case}

Note that each vertex $v\in V(H_{R})$, $d_{R}(v)\leq p-1$.
By the minimum-degree condition, for any $R_i\in\mathcal{R}$, $i\in[s]$,
$$p\left(\frac{p-2}{p-1}+o(1)\right)n\leq e(R_i,H_{\mathcal{R}})\leq \left(p-2+o(1)\right)n+\sum_{i=1}^{p}|N_i|.$$
Then 
$$\sum_{i=1}^{p}|N_i|\geq \left(\frac{p-2}{p-1}+o(1)\right)n.$$

Following the clique-switching argument at the end of Section 4 in \cite{2023}, we may choose $\mathcal{R}$ such that there exists a clique $R\in\mathcal{R}$, $N_j=\emptyset$ for some $j\in [p]$. Without loss of generality, we assume that 
$$N_p=\emptyset.$$
Note that each pair of vertices in $N_i$ is 
a $K_{p+1}$-saturating edge. 
Then 
\begin{align}
        f_{p+1}(G)&\geq\sum_{i=1}^{p-1}\scalebinom{|N_i|}{2}\notag\\
        &\geq (p-1)\scalebinom{\frac{1}{p-1}\sum_{i=1}^{p-1}|N_i|}{2}\notag\\
        &\geq \left(\frac{(p-2)^2}{2(p-1)^3}+o(1)\right)n^2.
    \end{align}


\begin{case}
    $s=\Theta(n).$
\end{case}

Applying Lemma 4.3 of \cite{2023} with $e(G)$ replaced by $\left(\frac{p-2}{2(p-1)}+o(1)\right)n^2$, we have
\begin{align}\label{fff1}
    f_1\geq \frac{p-2}{p-1}sn-\frac{p(p-2)}{2(p-1)}s^2+o(n^2).
\end{align}

    

By exploiting the minimum-degree condition, we can strengthen the lower bound in Lemma 4.2 of \cite{2023} as follows. 

\begin{lemma}\label{more}
    There exists a clique $R=\{v_1, v_2,..., v_p\}$  in $\mathcal{R}$ such that 
    \[
    e(R,H_{\mathcal{R}})\geq \frac{p(p-2)}{p-1}(n-ps)+o(n).
    \]
\end{lemma}

\begin{proof}
    By the minimum-degree condition,
$$2e\left(H_{\mathcal{R}}\right)+e\left(\mathcal{R},H_{\mathcal{R}}\right)=\sum_{v\in V(H_{\mathcal{R}})}d_{G}(v)\geq \frac{p-2}{p-1}n(n-ps)+o(n^2).$$
Combining (\ref{size}) with
$$e(G)=e\left(\mathcal{R}\right)+e\left(\mathcal{R},H_{\mathcal{R}}\right)+e\left(H_{\mathcal{R}}\right),$$
we have
\begin{align}
    e\left(\mathcal{R}\right)&\leq e(G)+e\left(H_{\mathcal{R}}\right)-\frac{p-2}{p-1}n(n-ps)+o(n^2)\notag \\
    &\leq  \frac{p-2}{2(p-1)}\left(n^2+(n-ps)^2-2n(n-ps)\right)+o(n^2)\notag \\
    &=\frac{p-2}{2(p-1)}p^2s^2+o(n^2).
\end{align}
Furthermore,
$$\sum_{i\in [s]}e\left(R_i,\mathcal{R}\backslash \{R_{i}\} \right)=2e(\mathcal{R})+o(n^2).$$
Hence, there exists $R^*\in \mathcal{R}$ such that 
$$e\left(R^*,\mathcal{R}\backslash \{R^*\} \right)\leq \frac{2e(\mathcal{R})}{s}\leq \frac{p-2}{p-1}p^2s+o(n)$$
Then
\begin{align*}
    e\left(R^*,H_{\mathcal{R}}\right)&=\sum_{v\in V(R^*)}d_{G}(v)-p(p-1)-e\left(R^*, \mathcal{R}\backslash \{R^*\}\right)\\
    &\geq \frac{p-2}{p-1}pn-\frac{p-2}{p-1}p^2s+o(n)\\
    &= \frac{p(p-2)}{p-1}(n-ps)+o(n).
\end{align*}
Hence, the result follows.
 \end{proof}
Combining Lemma \ref{more} with the discussion at the end of Section 4 in \cite{2023}, we may assume that there exists a clique $R\in \mathcal{R}$ such that 
\[     e(R,H_{\mathcal{R}})\geq \frac{p(p-2)}{p-1}(n-ps)+o(n),
    \]
    and $N_i=\emptyset$ for some $i\in [p]$.
Without loss of generality, we assume that $$N_p=\emptyset.$$

By using the approach of Lemma 4.4 in \cite{2023}, we provide a lower bound on $f_2$. 

\begin{lemma}\label{fff2}
    $f_2\geq\frac{(p-2)^2}{2(p-1)^3}(n-ps)^2+o(n^2).$
\end{lemma}
\begin{proof}
Recall that $e(R,H_{\mathcal{R}})\geq \frac{p(p-2)}{p-1}(n-ps)+o(n)$, and $N_p=\emptyset$.
    Since each vertex $v\in V(H_{R})$, $d_{R}(v)\leq p-1$,
    $$e(R,H_{\mathcal{R}})\leq (p-2)(n-ps)+\sum_{i=1}^{p-1}|N_i|.$$
    Therefore
    \begin{align*}
        \sum_{i=1}^{p-1}|N_i|&\geq \frac{p(p-2)}{p-1}(n-ps)-(p-2)(n-ps)+o(n)\\
        &=\frac{p-2}{p-1}(n-ps)+o(n).
    \end{align*}
     Note that each pair of vertices in $N_i$ is a 
     $K_{p+1}$-saturating edge. Then
    \begin{align}\label{f2}
        f_2&\geq\sum_{i=1}^{p-1}\scalebinom{|N_i|}{2}\notag\\
        &\geq (p-1)\scalebinom{\frac{1}{p-1}\sum_{i=1}^{p-1}|N_i|}{2}\notag\\
        &\geq \frac{(p-2)^2}{2(p-1)^3}(n-ps)^2+o(n^2).
    \end{align}
    Therefore, the lemma follows.
\end{proof}

By (\ref{fff1}) and Lemma \ref{fff2},
\begin{align}\label{combine}
   f_{p+1}(G)&=f_1+f_2\notag\\
    &\geq \frac{p-2}{p-1}sn-\frac{p(p-2)}{2(p-1)}s^2+\frac{(p-2)^2}{2(p-1)^3}(n-ps)^2+o(n^2)\notag\\
    &=\frac{(p-2)^2}{2(p-1)^3}n^2+\frac{(p-2)s(2n-ps)}{2(p-1)^3}+o(n^2)\notag\\
    &> \left(\frac{(p-2)^2}{2(p-1)^3}+o(1)\right)n^2,
\end{align}
for $p\geq 3$ and $\Theta(n)= s\leq \frac{n}{p}$. 

In conclusion, we have 
$$f_{p+1}(G)\geq \left(\frac{(p-2)^2}{2(p-1)^3}+o(1)\right)n^2.$$
Consequently, asymptotic equality can occur only when $s=o(n)$.

\section{Stability near $\mathrm{ex}(n,K_{p+1})-(c+o(1)) n^2 $}

In this section, we provide the proof of Theorem \ref{smain}. 



{\bf Proof of the upper bound}\quad
Denote by $G^*$ the graph obtained from the following construction.
We begin with the balanced complete $p$-partite Tur{\' a}n graph, whose vertex classes are denoted by $$V_1, V_2, \ldots, V_p, \ \ |V_i|=\frac{n}{p},\ i\in [p] .$$
For each $i\in [p]$, choose a distinguished vertex $a_i\in V_i$ and leave its adjacency relations unchanged. Define $U_i:=V_i\backslash \{a_i\}$, $|U_i|=\frac{n}{p}-1$.
Let $$U_i=\left\{u_0^i, u_1^i,\ldots, u_{\left(\frac{n}{p}-2\right)}^i\right\}.$$
For any $i,j\in [p]$ and $i\neq j$, we delete the edges of a $\frac{2c}{p-1}n$-regular graph between $U_i$ and $U_j$.
Specifically, remove the edges 
$$u_s^iu_{s+t}^j, \quad s=0,1,\ldots,\frac{n}{p}-2, \quad t=0,1,\ldots, \frac{2c}{p-1}n-1,$$
where the subscripts are taken modulo $\frac{n}{p}-1$.
That is, let $q=\frac{2c}{p-1}n$. For each vertex $u_0^i\in U_i$, we delete the edges between $u_0^i$ and $u_0^j,u_1^j,\ldots,u_{q-1}^j$; for each vertex $u_1^i\in U_i$, we delete the edges between $u_1^i$ and $u_1^j,u_2^j,\ldots,u_{q}^j$; for each vertex $u_2^i\in U_i$, we delete the edges between $u_2^i$ and $u_2^j,u_3^j,\ldots,u_{q+1}^j$; and so on. After this operation, each vertex in $U_i$ loses exactly $\frac{2c}{p-1}n$ neighbors in $U_j$. Similarly, every vertex in $U_j$ loses exactly $\frac{2c}{p-1}n$ neighbors in $U_i$.
Since $i$ and $j$ were chosen arbitrarily, we have
$$\delta(G^*)=\left(\frac{p-1}{p}-2c\right)n,$$
$$e(G^*)=\frac{1}{2}\left(\frac{p-1}{p}-2c+o(1)\right)n^2.$$

    For any $u,v\in V_i$, the graph $G^*+uv$ contains a copy of $K_{p+1}$. One such copy may be chosen with vertex set $\{a_1,a_2,..., a_{i-1},u,v,a_{i+1},\ldots,a_p \}$.
Hence
$$f_{p+1}(G^*)=\sum_{i=1}^{p}\scalebinom{|V_i|}{2}=\left(\frac{1}{2p}+o(1)\right)n^2$$

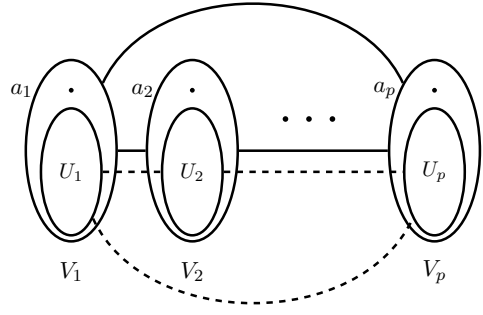
\begin{figure}[!h]
 	\centering
 	\scalebox{0.8}{
\begin{tikzpicture}[
  edge/.style={line width=2.5pt}
]

\coordinate (V2) at (-1,0);
\coordinate (V1) at (-3,0);
\coordinate (V3) at (3,0);
 \coordinate (U1) at (-3,-0.35);
 \coordinate (U2) at (-1,-0.35);
 \coordinate (U3) at (3,-0.35);
\node[ellipse,draw,thick,line width=1.2pt,minimum width=1.5cm,minimum height=3cm, font=\large] (v1) at (V1) {};
\node at ($(v1.center)+(0,-2)$) {$V_1$};
\node[ellipse,draw,thick,line width=1.2pt, minimum width=1.5cm,minimum height=3cm, font=\large] (v2) at (V2) {};
\node at ($(v1.center)+(2,-2)$) {$V_2$};
\node[ellipse,draw,thick,line width=1.2pt,minimum width=1.5cm,minimum height=3cm, font=\large] (v3) at (V3) {};
\node at ($(v1.center)+(6,-2)$) {$V_p$};
\node at (1,0.5)[font=\Huge] {$\cdots$};
\node[ellipse,draw,thick,line width=1.2pt,minimum width=1cm,minimum height=2.1cm, font=\small] (u1) at (U1) {$U_1$};
\node[ellipse,draw,thick,line width=1.2pt,minimum width=1cm,minimum height=2.1cm, font=\small] (u2) at (U2) {$U_2$};
\node[ellipse,draw,thick,line width=1.2pt,minimum width=1cm,minimum height=2.1cm, font=\small] (u3) at (U3) {$U_p$};
\node[circle,fill,inner sep=0.8pt] (w1) at (-3,1) {};
\node[left] at ($(w1)+(-0.5,0)$) {$a_1$};
\node[circle,fill,inner sep=0.8pt] (w2) at (-1,1) {};
\node[left] at ($(w2)+(-0.5,0)$) {$a_2$};
\node[circle,fill,inner sep=0.8pt] (w3) at (3,1) {};
\node[left] at ($(w3)+(-0.5,0)$) {$a_p$};

 \draw[line width=1.2pt] (v1)--(v2);
 \draw[line width=1.2pt] (v2) to (v3);
 \draw[line width=1.2pt] (v1) to [bend left=65] (v3);
 \draw[dashed, line width=1.2pt] (u1)--(u2);
  \draw[dashed, line width=1.2pt] (u2)--(u3);
 \draw[line width=1.2pt,dashed] (u1) to [bend right=65] (u3);


\end{tikzpicture}}
 		\caption{Extremal graph for general $p\geq 3$}
         \label{Sgraph}
 	
 \end{figure}

{\bf Proof of the lower bound}\quad
Let $G$ be an $n$-vertex $K_{p+1}$-free graph,
$$e(G)=ex(n,K_{p+1})-\left(c+o(1)\right)n^2, \quad \delta(G)\geq \left(\frac{p-1}{p}-2c+o(1)\right)n, $$
where $0\leq c< \frac{1}{2p(3p-1)}$.
Then, by Tur{\'a}n theorem, $G$ contains at least one copy of $K_p$.
Moreover, $$\delta(G)>\frac{3p-4}{3p-1}n,$$
for $c<\frac{1}{2p(3p-1)}$.

Next, we present the following lemma, which provides structural restrictions for graphs under forbidden subgraph and minimum degree conditions and serve as a key tool in our subsequent analysis. 

\begin{lemma}[{\rm Andr{\' a}sfai-Erdős}, \cite{1974}] \label{partite}
    Let $p\geq 3$. For a arbitrary $n$-vertex $K_p$-free graph $G$, if $\delta(G)>\frac{3p-7}{3p-4}n$, then $G$ is a $(p-1)$-partite graph.
\end{lemma}

By Lemma \ref{partite}, $G$ is a $p$-partite graph.
Then let $$V(G)=V_1\cup V_2\cup\cdots\cup V_p,$$ where each $V_i$ is independent for $i\in [p]$.
Note that, for any $t\in [p]$, $$|V_t|\geq \left(\frac{1}{p}-2c(p-1)+o(1)\right)n.$$
Indeed, let $j\in [p]$ and take any vertex $v\in V_j$. Then $d_G(v)\leq n-|V_j|$. 
It follows that $$|V_j|\leq\left(\frac{1}{p}+2c+o(1)\right)n.$$ 
As the choice of $V_j$ was arbitrary, we further obtain that 
\begin{align}\label{Vsize}
    |V_t|\geq n-(p-1)\left(\frac{1}{p}+2c+o(1)\right)n=\left(\frac{1}{p}-2c(p-1)+o(1)\right)n.
\end{align}

We next prove that every non-edge inside any partite set is a $K_{p+1}$-saturating edge.

\begin{claim}\label{induce}
    For any $i\in [p]$ and any two vertices $u,v\in V_i$, $uv\notin E(G)$. Then $G+uv$ contains a copy of $K_{p+1}$.
\end{claim}
\begin{proof}
    Without loss of generality, assume that $u,v\in V_p$.
    For every $i\neq p$, define $$A_i=V_i\cap N_G(u)\cap N_G(v).$$
    We first prove that $A_i$ has linear size.
    Combining the minimum-degree condition with (\ref{Vsize}),
    \begin{align*}
        |A_i|&\geq d_G(u,V_i)+d_G(v,V_i)-|V_i|\\
        &\geq 2\left(\left(\frac{p-1}{p}-2c+o(1)\right)n-\sum_{\ell \neq i,p}|V_{\ell}| \right)-|V_i|\\
        &\geq 2|V_p|+|V_i|-2\left(2c+\frac{1}{p}+o(1)\right)n\\
        &\geq \left(\frac{1}{p}-2c(3p-1)+o(1)\right)n.
    \end{align*}
Since $c<\frac{1}{2p(3p-1)}$, we have $\frac{1}{p}-2c(3p-1)>0$. Hence, $|A_i|=\Omega(n)$.

We now proceed greedily.
Assume that vertices $$w_1,w_2,\ldots,w_k$$ have already been chosen, where $k\in [p-2]$ and $w_i\in A_i$ for each $i\in [k]$.
Our goal now is to select a vertex $w_j\in A_j$, where $j=k+1$, that is adjacent to all previously selected vertices.
We claim that such a choice is always possible.
We next show that the candidate set $$A_j\cap \left(\bigcap_{w_i\in \{w_1,w_2,\ldots,w_k\}}N_G(w_i)\right)$$ is nonempty.
Since 
$$d_G(w_i,V_j)\geq\left(\frac{p-1}{p}-2c+o(1)\right)n-\sum_{\ell\neq i,j}|V_\ell|,$$
$$\sum_{\ell\neq i,j}|V_\ell|=n-|V_i|-|V_j|,$$
we have $$|V_j|-d_G(w_i,V_j)\leq \left(\frac{1}{p}+2c+o(1)\right)n-|V_i|.$$
Moreover, for any $i\in [k]$, there are $|V_j|-d_G(w_i,V_j)$ vertices in $V_j$ that are not adjacent to $w_i$.

Then, combining $k\leq p-2$ with 
$$|A_j|\geq 2|V_p|+|V_j|-2\left(\frac{1}{p}+2c+o(1)\right)n,$$
we have
\begin{align*}
    \left|A_j\cap \left(\bigcap_{w_i\in \{w_1,w_2,\ldots,w_k\}}N_G(w_i)\right)\right|&\geq 2|V_p|+|V_j|-2\left(\frac{1}{p}+2c+o(1)\right)n\\
    &\ \ \ \ -\left((p-2)\left(\frac{1}{p}+2c+o(1)\right)n-\sum_{i\neq p,j}|V_i|\right)\\
    &= n+|V_p|-p\left(\frac{1}{p}+2c+o(1)\right)n.
\end{align*}
By (\ref{Vsize}),
\begin{align*}
    \left|A_j\cap \left(\bigcap_{w_i\in \{w_1,w_2,\ldots,w_k\}}N_G(w_i)\right)\right|&\geq n+\left(\frac{1}{p}-2c(p-1)+o(1)\right)n-p\left(\frac{1}{p}+2c+o(1)\right)n\\
    &=\left(\frac{1}{p}-2c(2p-1)+o(1)\right)n.
\end{align*}
Since $$\frac{1}{p}-2c(2p-1)>0$$ for $c<\frac{1}{2p(3p-1)}$,
we can successively choose this vertex set $$\{w_1,w_2,\ldots,w_{p-1}\},\quad w_i \in A_i, \ \text{for} \ i\in [p-1],$$ such that the selected vertices are pairwise adjacent.
Therefore, the vertex set $$\{u,v\}\cup\{w_1,w_2,\ldots,w_{p-1}\}$$
induces a copy of $K_{p+1}$ in $G+uv$.
\end{proof}

 By Claim \ref{induce}, we have every non-edge inside $V_i$ is a $K_{p+1}$-saturating edge, $i\in [p]$.
 Hence,
 \begin{align}
     f_{p+1}(G)\geq \sum_{i=1}^{p}\scalebinom{|V_i|}{2}\geq p\scalebinom{\frac{1}{p}\left(\sum_{i=1}^{p}|V_i|\right)}{2}=\left(\frac{1}{2p}+o(1)\right)n^2.
 \end{align}
     
This completes the proof.

\section{Concluding remarks}
  (1)  The factor $\dfrac{p-2}{p-1}$  featured in the minimum-degree condition from Theorem \ref{main} is the best possible.  

 For any $\varepsilon>0$, let $\alpha_p= \frac{(p-2)^2}{2(p-1)^3}, $ and  $0<\theta<\min\left\{\frac{2}{p}, \frac{4\alpha_p}{2+\alpha_p},\frac{(p-1)\varepsilon}{p-2}\right\}.$ 
We investigate the graph $G_\varepsilon$ with $\delta(G_\varepsilon)\geq \left(\frac{p-2}{p-1} -\varepsilon+o(1)\right)n$ in the following. Define
\[
\beta = \frac{(p-2)(2-\theta)}{2(p-1)^2}, \qquad 
\gamma = \frac{2 - p\theta}{2(p-1)^2}.
\]
Since $\theta < \frac{2}{p}$, we have $\gamma > 0$. Moreover,
$
\theta + (p-1)(\beta + \gamma) = 1.
$
Consider the construction $G_\varepsilon$ in Figure 1 with $ |V_0|=\theta n+o(n), \  
  |V_i|=\beta n+o(n), \   |U_j|=\gamma n+o(n) \  \text{for}\quad  i,j\in [p-1].$ 
Hence,
$$e(G_\varepsilon) = \operatorname{ex}(n, K_p) + o(n^2), $$ and, for $v\in U_j$, $j\in [p-1]$, $$
\delta(G_\varepsilon) =d_G(v)= \left(\frac{p-2}{p-1} (1-\theta)+o(1)\right)n \ge \left(\frac{p-2}{p-1}  - \varepsilon +o(1)\right)n,$$
which follows from
$\theta < \frac{p-1}{p-2}\varepsilon$. Moreover,
\begin{align*}f_{p+1}(G_\varepsilon) &=  \sum_{i=0}^{p-1} \binom{|V_i|}{2} + o(n^2)\\
&=\left( \frac{\theta^2}{2} + \alpha_p  (1 - \frac{\theta}{2} )^2 + o(1) \right) n^2\\
&<\left( \frac{(p-2)^2}{2(p-1)^3} - o(1) \right) n^2,
\end{align*}
where the last inequality holds because  $0 < \theta < \frac{4\alpha_p}{2 + \alpha_p}$.

(2)
In the context of spectral extremal graph theory, we also consider the stability of the $t$-clique spectral radius. For a graph $G$, let $\rho_t(G)$ denote the $t$-clique spectral radius, defined as the spectral radius of the $t$-uniform adjacency tensor associated with $G$. More precisely, let $\mathcal{A}_t(G)=\left(a_{i_1i_2\cdots i_t}\right)$ be the adjacency tensor of $G$, where 
\begin{align}
a_{i_1i_2\ldots i_t}=
    \begin{cases}
        &\frac{1}{(t-1)!}, \ \text{if} \{v_{i_1}, v_{i_2},\ldots, v_{i_t}\} \text{induces a}\ K_t, \\
       & 0,\ \text{otherwise.} 
    \end{cases}
\end{align}
Then $\rho_t(G)$ is defined as the largest eigenvalue of $\mathcal{A}_t(G)$.

Let $p\geq 2$ be an integer, and let $t\in \{2,3,\ldots, p\}$ be fixed. Let $n$ be a sufficiently large integer. Denote by $G$ an $n$-vertex $K_{p+1}$-free graph.  If
$$\rho_t(G)\geq \rho_t(T_p(n))-\eta_nn^{t-1}, \quad \eta_n \rightarrow 0, \quad \eta_n\geq 0,$$
then 
$$f_{p+1}(G)\geq \left(\frac{1}{2p}+o(1)\right)n^2.$$
Moreover, this bound is asymptotically sharp.

The proof is based on Theorem 1.5 in \cite{2026}.
By applying Theorem 1.5, there exists a partition
$$V(G)=V_1\cup V_2\cup\cdots\cup V_p$$
such that 
$$|V_i|=\left(\frac{1}{p}+o(1)\right)n, \quad \text{for each} \ i\in [p],$$
and
$$\sum_{i=1}^{p}e\left(G[V_i]\right)+\sum_{1\leq i<j\leq p}\left(|V_i||V_j|-e(V_i,V_j)\right)=o(n^2).$$

After removing $o(n)$ exceptional vertices, every remaining vertex has at most $o(n)$ non-neighbors in some other partition. 
Applying the same greedy embedding strategy as in the proof of Theorem \ref{smain}, we conclude that every non-edge inside each remaining part is a $K_{p+1}$-saturating edge.
Consequently,
$$f_{p+1}(G)\geq  \sum_{i=1}^{p}\scalebinom{|V_i|}{2}\geq \left(\frac{1}{2p}+o(1)\right)n^2.$$

Moreover, the Tur{\' a}n graph $T_p(n)$ satisfies the spectral condition with equality, and its $K_{p+1}$-saturating edges are exactly the pairs lying  within the same part. Therefore, this proves the asymptotic sharpness.

\section{Acknowledgments}
The authors would like to express their sincere gratitude to Professor Hong Liu for proposing the problem investigated in this paper and for his valuable insights into the expected results, particularly in the case $p=3$.

\end{document}